\documentclass{amsart}

\usepackage{ amsmath, amssymb, amsthm, hyperref, graphicx,   verbatim, enumerate,xcolor}
\usepackage{mathabx}
\usepackage[paper=a4paper]{geometry}
\usepackage[utf8]{inputenc}
\usepackage[T1]{fontenc}

\newcommand{\cc}{\mathbb{C}}

\newcommand{\re}{\mathbb{R}}

\newcommand{\e}{\varepsilon}

 \newcommand{\cW}{\mathcal{W}}

\newcommand{\fr}{\partial}

\newcommand{\set}[1]{\left\{#1\right\}}
\newcommand{\norm}[1]{\left\Vert#1\right\Vert}
\newcommand{\abs}[1]{\left\vert#1\right\vert}
\newcommand{\cd}{{\cc^2}}

\newcommand{\pu}{{\mathbb{P}^1}}

\newcommand{\rest}[1]{ \arrowvert_{#1}}

\newcommand{\unsur}[1]{\frac{1}{#1}}
\newcommand{\el}{\mathcal{L}}

\newcommand{\lrpar}[1]{\left(#1\right)}

\newcommand{\la}{\lambda}

\newcommand{\loc}{\mathrm{loc}}

\newcommand{\inv}{^{-1}}

\newcommand{\jstar}{J^\varstar}

\DeclareMathOperator{\supp}{Supp}

\DeclareMathOperator{\jac}{Jac}
\DeclareMathOperator{\dist}{dist}

\newtheorem{prop}{Proposition} [section]
\newtheorem{thm}[prop] {Theorem}
\newtheorem{defi}[prop] {Definition}
\newtheorem{lem}[prop] {Lemma}
\newtheorem{cor}[prop]{Corollary}

\newtheorem{question}[prop]{Question}
\newtheorem{conjecture}[prop]{Conjecture}

\theoremstyle{remark}

\newtheorem{rmk}[prop]{Remark}

\begin{document}


\title[Hyperbolicity criteria for automorphisms of $\cd$]{Topological and geometric hyperbolicity criteria for polynomial automorphisms of $\cd$}
\author{Eric Bedford}
\address{Institute for Mathematical Sciences Stony Brook University Stony Brook, NY 11794}
\email{ebedford@math.stonybrook.edu}

\author{Romain Dujardin}
\address{ Sorbonne Université, CNRS, Laboratoire de Probabilités, Statistiques et Modélisations (LPSM), F-75005 Paris, France}
\email{romain.dujardin@sorbonne-universite.fr}

 \begin{abstract}
We prove that uniform hyperbolicity is invariant under topological conjugacy for dissipative polynomial automorphisms 
 of $\mathbb{C}^2$. 
Along the way we also show that a sufficient condition for hyperbolicity is that local 
stable and unstable manifolds of saddle points have uniform geometry. 
\end{abstract}

 \maketitle 

\section{Introduction}

\subsection{} The main motivation  of this note is to study the following problem:

\begin{question}\label{q:basic}
Is uniform hyperbolicity a topological property for complex Hénon maps? 
\end{question}

We use the terminology 
  ``complex Hénon map'' as a synonym  for ``polynomial automorphism of $\cd$ with non-trivial dynamics''. 
By \cite{FM} we can normalize such a map such that it is a product of actual  
Hénon maps $(z,w)\mapsto (aw+p(z) , az)$.  
Hyperbolicity here is understood in the sense of \cite{bs1}, that is, 
we say that a complex Hénon map $f$  is \emph{hyperbolic} if its Julia set 
  $J = J_f$ is a hyperbolic set, which must then be of saddle type. This was  shown in \cite{bs1} to have 
   strong  consequences on the global dynamics of $f$: 
 the chain recurrent set consists of  $J$ together with finitely many periodic attractors, and $f$ satisfies Smale's Axiom A on $\cd$.  
 (See   \cite{ishii survey} for a recent survey on complex Hénon dynamics, with an emphasis on hyperbolic maps.)
  
  Let us recall a bit of     standard notation. Given a complex Hénon map $f$, 
  we denote by $J^+$   the forward  Julia set, which is 
  the locus of non-normality of forward iterates $(f^n)_{n\geq 0}$, or equivalently 
  the boundary of the set $K^+$ of points with bounded forward orbits. We likewise define $J^-$ and $K^-$ for backward dynamics, 
  and we set $J = J^+\cap J^-$. We denote by $\jstar$ the closure of the set of saddle periodic orbits, which is contained in $J$. 
  For hyperbolic maps, we have  $J=\jstar$; however, 
   this equality is an open problem in the general case.

Question \ref{q:basic} was previously considered for   rational maps on the Riemann sphere and for  smooth
 Anosov diffeomorphisms of (real) compact manifolds. Let us  start by briefly reviewing these cases. 
  
\subsection{One-dimensional rational maps} 
 For polynomials and rational maps in one variable, the answer to the question is ``yes'' since there is a simple topological 
criterion for hyperbolicity: $\overline{ \mathrm{PC}(f)}\cap J = \emptyset$, where $\mathrm{PC}(f)$ is the postcritical set. 
As a consequence, 
 if $f_0$ and $f_1$ are rational maps in one variable such that $f_0$ is hyperbolic, and $\phi: \pu\to\pu$ is a
 topological conjugacy between $f_0$ and $f_1$, then $f_1$ is hyperbolic. Actually the  statement already 
 holds locally near the Julia set:
 
 \begin{prop}\label{prop:rational}
 Let  $f_0$ and $f_1$ be rational maps on $\pu$, and assume that $f_0$ is hyperbolic.  If there are neighborhoods
 $N(J_0)$ and $N(J_1)$ of $J_0$ and $J_1$ and a homeomorphism $\phi:N(J_0)\to N(J_1)$ which satisfies 
 $\phi\circ f_0=f_1\circ \phi$ wherever these compositions make sense, then
 $f_1$ is hyperbolic. 
 \end{prop}

The result is not obvious    since the conjugacy $\phi$ cannot  detect that  a point in $N(J_1)$ is post-critical.  
Here and throughout the paper we use indices 0 and 1 to label the dynamical objects (Julia set, etc.) respectively associated to $f_0$ and $f_1$. 

\begin{proof}   
Notice first that the conjugacy $\phi$ sends  periodic points to
periodic points. The topological dynamics around a periodic point 
determines its type (attracting, repelling, neutral) so  it follows that all periodic points of $f_1$ are hyperbolic. 
In particular $f_1$ has no parabolic points. 

Since the Julia set is the accumulation set of periodic orbits   we infer that 
$\phi(J_0) = J_1$. 
Now a rational map without parabolic points is hyperbolic 
if and only if its critical set is disjoint from the Julia set. 
This property holds for $f_1$ by the topological conjugacy, so the result follows. 
\end{proof}

\begin{rmk} \label{rmk:rational} \
\begin{enumerate}
\item It is not enough in the proposition  to assume that $\phi$ is a conjugating 
homeomorphism $J_0\to J_1$. Indeed it is well known that 
$z^2+ \frac14$ is topologically conjugate on its Julia set to any quadratic polynomial in the main cardioid 
(e.g. $z^2$). As we will comment below, a similar  phenomenon holds
for complex Hénon maps (see \cite{radu tanase}). 
\item If we suppose a priori that $\deg(f_0) = \deg(f_1) = d$ we can relax the assumption on $\phi$ by assuming  only 
that 
$\phi$ is any injective continuous map defined in a neighborhood $N_0$ 
of $J_0$ and such that $\phi\circ f_0 = f_1\circ \phi$ wherever 
these  compositions makes sense. Indeed  
by the invariance of domain theorem, $\phi(N_0)$ is an open subset of the plane. We only have to shows that it
  contains   $J(f_1)$. Indeed $f_0$ has only finitely many non-repelling periodic points, so 
  $J_0$ contains $p_n\sim d^n$ repelling periodic points of period $d$ for large $n$.  Thus by the topological conjugacy, 
  $\phi(J_0)$  contains  $p_n$ repelling periodic points of $f_1$, which are equidistributed to the equilibrium measure 
  $\mu_{f_1}$ whose support is $J(f_1)$. Therefore   $\phi(J_0) = J_1$ and we are done. 
 \end{enumerate}
\end{rmk}

\subsection{Anosov diffeomorphisms} \label{subs:anosov}
The problem of topological invariance of hyperbolicity  
 in real dynamics has been  popularized in particular by A. Katok. 
 The answer is already quite subtle for the simplest case of Anosov diffeomorphisms of the 2-torus. 
  
Indeed  there exist examples of  $C^2$
diffeomorphisms $f$ of the 2-torus which are not hyperbolic but still    globally 
topologically conjugate to a linear Anosov map. 
This can be done by  either  carefully deforming a linear Anosov map 
 until some saddle fixed point becomes neutral  by preserving the geometry of the stable and unstable foliations 
 (see \cite{katok 79}),
 or by deforming the foliations until 
 reaching a cubic  heteroclinic tangency  
 (see \cite{enrich, bonatti diaz vuillemin}).

If we now impose   the conjugacy to be H\"older then there are different regimes depending on    the 
precise H\"older regularity. First, it can be arranged that 
in the previous examples the conjugacy and its inverse are 
Hölder continuous  \cite{gogolev}, and thus hyperbolicity is not
invariant under H\"older conjugacy. On the other hand if the conjugacy is sufficiently close to 
being bi-Lipschitz --namely, the product of the H\"older exponents of $\phi$ and $\phi\inv$ is larger than 1/2-- then $f$ 
is Anosov \cite{fisher} (see \cite{gogolev}).

%

 \subsection{A conjecture}
 The most natural way to address  Question \ref{q:basic} would be to find 
 a   topological criterion 
 ensuring hyperbolicity for a complex Hénon map, in the spirit of the one-dimensional condition
 $J\cap \overline { \mathrm{PC}(f)} = \emptyset$. 
 Strictly speaking, a Hénon map admits 
 no critical points; nevertheless there are  ways to give a reasonable meaning to this condition 
 -- which are more differential-geometric than topological, though.  For instance, in the dissipative regime,
  the condition that there are no critical points 
 on $J$ naturally corresponds to the existence of a dominated splitting and,  provided dissipation is strong enough,
  a good analogue of the one-dimensional situation  was achieved in  \cite{lyubich peters}. 
  
 A variant is to translate  the condition $J\cap \overline { \mathrm{PC}(f)} = \emptyset$ into a regularity property of the geometry 
 of the forward and backward Julia sets $J^{+/-}$ near $J$. In this respect it was shown in \cite{bs8} that if 
 in some neighborhood of $J$, $J^+$ and $J^-$ 
 are the supports  of two    
 Riemann surface laminations  which are transverse along   $J$, then $f$ is hyperbolic. 
In  \S\ref{sec:geometric} below we reprove and generalize this result in several ways. 

 Back to our initial problem, even if it is unclear how to design  a purely topological criterion for hyperbolicity, one 
 may ask whether hyperbolicity is invariant under topological conjugacy. 
Here is a precise analogue of Proposition \ref{prop:rational} for complex Hénon maps:
 
 \begin{conjecture} \label{conj:topological}
 Let $f_0$ and $f_1$ be two polynomial automorphisms of $\cd$ with non-trivial dynamics, and assume that $f_0$ is hyperbolic. 
 Suppose that there exists  respective  neighborhoods $N_0$ and $N_1$ 
 of $J_0 = \jstar_0$ and $\jstar_1$ and a homeomorphism $\phi: N_0\to N_1$ such that $\phi\circ f_0 = f_1\circ \phi$ wherever 
these  compositions makes sense. Then $f_1$ is hyperbolic. 
 \end{conjecture}

Here are some comments on this conjecture: 
\begin{enumerate}
\item It was shown in \cite{saddle} (see also \cite{guerini peters})
that for a complex Hénon map hyperbolicity on $\jstar$ implies hyperbolicity. This explains why we can restrict
to a neighborhood of $\jstar_0$ instead of $J_0$, and opens the way to an analysis of hyperbolicity based on periodic points. 
\item If we add the hypothesis that $f_0$ and $f_1$ have the same dynamical degree, then 
by using the equidistribution of periodic orbits from \cite{bls2} and arguing as 
in Remark \ref{rmk:rational}  we can relax the assumption that 
$\phi(N(\jstar_0))$ contains $\jstar_1$. 
\item As observed in Remark \ref{rmk:rational}, the conjecture is false  if the conjugating 
homeomorphism $\phi$ is  only defined on $J_0=\jstar_0$ (see  \cite{radu tanase}). 
\item The conjecture is true if $\phi$ is obtained by deformation in the following sense:
  it was shown in \cite{hyperbolic} that if there is a weakly   stable holomorphic family 
$(f_{\la})$ connecting $f_0$ and $f_1$, then $f_1$ is hyperbolic\footnote{The original statement in \cite{hyperbolic} concerns hyperbolicity on $\jstar$, so we further use \cite{saddle} to deduce
 hyperbolicity on $J$.}.
\end{enumerate}

\subsection{Quasi-hyperbolicity}

The methods in this paper are closely related to the notion of quasi-hyperbolicity. If $p$ is a saddle point and $r>0$, we let 
$W^{s/u}_r(p)$ denote the connected component of $W^{s/u}(p)\cap B(p,r)$ containing $p$. Following \cite{bs8}
 map $f$ is said
\emph{quasi-hyperbolic} if there exists positive constants $r$ and $B$ such that for every saddle periodic point $p$:
\begin{enumerate}[{ (i)}]
\item $W^{s/u}_r(p)$ is closed in $B(p,r)$ and
\item the area of $W^{s/u}_r(p)$ is bounded by $B$. 
\end{enumerate}
If $\phi:N_0\to N_1$ is a topological conjugacy as in Conjecture \ref{conj:topological} then $\phi$ preserves stable and unstable 
manifolds, so if {(i)} holds for $f_0$ it will also hold for $f_1$ (after possibly shrinking $r$). 
It was shown in \cite{bgs} that if $f$ is quasi-hyperbolic then there exist stable and unstable manifolds $\cW^{s/u}(x)$
 through each point   $x\in \jstar$. Furthermore $f$ is uniformly 
hyperbolic (on $\jstar$ and thus $J$) if and only if there is no tangency between $\cW^s$ and $\cW^u$. 
Thus  if we  know that $f$ is already quasi-hyperbolic, then the additional condition of  hyperbolicity is a topological
 invariant in the sense of the conjecture. At this stage, however, 
  it remains an open question whether quasi-hyperbolicity is a topological 
 property. 

\subsection{Results and outline}
In \S\ref{sec:geometric} we establish several  sufficient conditions for hyperbolicity based on the geometry of 
local stable and unstable manifolds of saddle periodic points. 
 A first sufficient condition for hyperbolicity, which essentially follows from \cite{bs8},
  is that these local stable and unstable manifolds have uniform size 
and the angle between them is uniformly bounded 
from below.   We give a self-contained proof of this result  (see Theorem \ref{thm:bs8_revisited}). 
We further show that the transversality 
assumption is superfluous (Theorems \ref{thm:uniform regularity} and \ref{thm:us regularity}), and that, as it might be expected,
  in the dissipative case it is enough to control the geometry of unstable manifolds 
   (Theorem \ref{thm:uuregular global}). 

In \S\ref{sec:topological} we prove   Conjecture \ref{conj:topological} in the case where $f_1$ is 
dissipative (Theorem \ref{thm:topological dissipative}).  In the conservative case the conjecture holds provided $\phi$ is H\"older continuous (Theorem \ref{thm:topological conservative}). The general case remains open\footnote{Notice that the Jacobian 
is not invariant under topological conjugacy: the Hénon map $(z,w)\mapsto (z^2+c + aw, z)$ is conjugate to a horseshoe for any Jacobian $a$,
when $ \abs{c}\gg \abs{a}$.}. 
    
  \section{Geometric criteria for hyperbolicity} \label{sec:geometric}

\subsection{Size of a submanifold at a point and u/s regularity} \label{subs:size}
Endow $\cd$ with the Euclidean metric.   A  {\em bidisk of size $r$} is the  image of $D(0, r)^2$ under  some affine 
isometry. 
A curve $V$ in  $\cd$ is a graph over an affine line $L$ if the orthogonal projection onto $L$ is injective when restricted to $V$.
Then there is  a well-defined notion of slope of a holomorphic curve with respect to $L$.  

\begin{defi}
A curve $V$ through $p$ is said 
to have  {size $r$ at $p$}
if there exists a neighborhood of $p$ in $V$ that is a graph of slope at most 1   
over a disk of radius $r$ in the tangent space $T_pV$. 
\end{defi}

If $\Delta$ be a disk of size $r$ at $p$,   fixing orthonormal coordinates $(x,y)$ so that $p=0$ and $T_pV = \set{y=0}$, we get that 
 the connected component of $\Delta$ through $p$ in the bidisk $ D(0,r)^2$ is a graph $\set{y=\varphi(x)}$
 over the first coordinate with $\abs{\varphi'}\leq 1$ and $\varphi'(0)=0$. 
 In particular if $\Delta$ is  immersed and has size $r$ at $p$, then   it is a submanifold in  $B(0, r/\sqrt{2})$ (because  a bidisk of size $r$ contains a ball of radius $r/\sqrt{2}$). 

 We now recall a few concepts from \cite{hyperbolic}. 
   A point $x\in \jstar$ is said {\em  u-regular} (resp. {\em  s-regular}) 
if there exists $r>0$ and a sequence of 
saddle points $(p_n)$ converging to  $x$ such that $W^u(p_n)$ (resp. $W^s(p_n)$) is of size $r$ at $p_n$. 
In this case it can be shown that the sequence of disks  $W^u_r(p_n)$ (resp. $W_r^s(p_n)$) converges 
in the $C^1$ topology to a (smooth) holomorphic disk of size $r$ at $x$ which we denote by $\cW^u_r(x)$ (resp. $\cW^s_r(x)$) (see 
\cite[Prop. 4.2]{hyperbolic}). This notation is meant to emphasize that 
at this stage $\cW^{s/u}_r(x)$ need not be an stable/unstable manifold in the usual sense. 
We  use the notation $\cW^{u/s}_\loc(x) $ for an unspecified neighborhood of 
$x$ in $\cW^u_r(x)$. 
We say that $x$ is {\em regular} if it is u- and s-regular and  $\cW_\loc^u (x)$ and $\cW_\loc^s (x)$ do not coincide, and 
 {\em transverse regular} if they are transverse. In particular we have the implications  
 $$\text{ u- and s-regular} \Leftarrow  \text{ regular } \Leftarrow\text{ transverse regular.}$$ 

It is easy to see that if $x$ is a saddle point, then $x$ is regular (for instance because it generates homoclinic intersections, hence it belongs to a horseshoe) and $\cW^s_\loc(x)$ and $\cW^u_\loc(x)$ coincide with the classical local stable and unstable manifolds 
of $x$. 

We define a local stable set $$\dot W^s_{\text{loc},\epsilon}(x) = \{y:\text{dist}(f^n(x),f^n(y))<\epsilon \  \forall n\ge0,\text{ and } \lim_{n\to\infty}\text{dist}(f^n(x),f^n(y))=0\}$$

\begin{lem}
Suppose that there is a complex disk $\Delta$ such that $x\in\Delta\subset\dot W^s_{\text{loc},\epsilon}(x)$.  If $x$ is s-regular, then $\Delta$ coincides with $\cW^s(x)$ locally at $x$.  The analogous result holds for `s' replaced by `u'.
\end{lem}

\begin{lem}
If $x$ is Pesin regular, and $x$ is u- and s-regular, then $\cW^{s/u}_{\text{loc}}(x)$ agree locally at $x$ with the Pesin manifolds $W^{s/u}_{\text{Pesin}}$.  Further, $\cW^s_{\text{loc}}(x)\ne \cW^u_{\text{loc}}(x)$, so $x$ is regular in the sense defined above, and in fact it is transverse regular.
\end{lem}

More generally we have:

\begin{lem}\label{lem:pesin_regular}
If $x$ is a Pesin regular point which is 
 regular, then
$\cW^s_\loc(x)$ and $\cW^u_\loc(x)$ coincide with the classical Pesin 
local stable and unstable manifolds of $x$. 
\end{lem}

\begin{proof} Denote temporarily the  Pesin 
local unstable manifolds  by $W^{u}_{\mathrm{Pesin}}(x)$.
If $(p_n)$ is a sequence of saddle points converging to $x$, then $\cW^u(p_n) =  W^u_\loc (p_n)$ must coincide with or
be disjoint from $W^u_{\mathrm{Pesin}}(x)$, and converge 
to $\cW^u(x)$ in the $C^1$ topology. Since both $W^u_{\mathrm{Pesin}}(x)$ and $\cW^u(x)$ contain $x$, by the 
Hurwitz theorem we conclude that $W^u_{\mathrm{Pesin}}(x)$ locally 
coincides with $\cW^u(x)$. 
\end{proof}

   We say that 
$x\in \jstar$ {\em  uniformly u-regular}    if the  uniform size
 property for local unstable manifolds holds for {\em any} sequence $(p_n)$ converging to  $x$.   
  If required we can specify the size $r$ in the 
 terminology. 
 Uniform s-regularity is defined similarly.  We say that $x$ is uniformly (resp. transverse) 
  regular if it is uniformly u- and s- regular, and  (resp. transverse) regular. 
  
The following result will play an important role in this paper (of course it admits an identical    s-regular version). 

\begin{prop}\label{prop:lamination}
The following assertions are equivalent:
\begin{enumerate}
\item Every point in $\jstar$ is uniformly $u$-regular.
\item There exists a uniform  $r>0$ such that for every saddle periodic point $p$, $W^u(p)$ has size $r$ at $p$.
\item There exists a uniform  $r>0$ and a dense set $D$ of saddle 
periodic points such that for every   $p\in D$, $W^u(p)$ has size $r$ at $p$.
\item There exists a lamination $\cW^u$  by Riemann surfaces in  a
neighborhood of $\jstar$ which extends the family of 
local unstable manifolds of saddle points.  
\end{enumerate}
\end{prop}

\begin{proof}
(1) implies (2) by a simple compactness argument. The implications (2)$\Rightarrow$(3) and (4)$\Rightarrow$(1) are obvious, and   
(3)$\Rightarrow$(4) follows from   standard holomorphic motions techniques. Let us give some details on this last point for the 
reader's convenience (see also Prop. 4.2 and Lemma 5.3 in  \cite{hyperbolic}). We start with the following basic geometric idea:  
 if $\Delta$ and $\Delta'$ 
flat disks in $B(0, r/10)\subset \cd$ which  both intersect $B(0, r/1000)$, and whose tangent vectors are  $1/4$-far apart  (relative to the usual Fubini-Study metric on $\pu$), then $\Delta$ and $\Delta'$ intersect and 
$$\dist(\Delta\cap \fr B(0, r/10), \Delta'\cap \fr B(0, r/10) )> \frac1{100}.$$ By the persistence of proper intersections, the same holds for $\widetilde \Delta$ and $\widetilde \Delta'$, whenever $\widetilde \Delta$ and $\widetilde \Delta'$ are holomorphic disks which are respectively 
 $1/100$ close to $\Delta$ and $\Delta'$. Now if $\Delta$ is a disk of size $r$ at $x$, by the Schwarz Lemma, $\Delta\cap B(0, r/10)$ 
 remains $1/100$ close to $T_x\Delta$. Taking the contraposite we see that 
  if $\Delta$ and $\Delta'$ are disks of size $r$ respectively at $x$ and $x'$, with 
 $\dist(x,x')<1/1000$, then 
 their tangent spaces must be $1/4$-close to each other, in particular they are graphs over a disk of radius $r/4$
relative to the same orthogonal projection.  

  Now by (3), for every $x\in \jstar$ there is a  holomorphic disk $\cW^u_r(x)$ 
  of size $r$ through $x$   and 
  these disks are either disjoint or locally coincide  because  local unstable manifolds of saddle points are disjoint. 
  By the previous discussion, the disks 
  $\cW^u_{r/4}(y)$ are disjoint graphs over some direction for $y$ close to $x$, 
so they form a lamination by 
the Lambda Lemma of \cite{mss}. Thus we get the desired lamination structure 
in the $r/5$-neighborhood of $\jstar$. 
\end{proof}

 \begin{rmk}\label{rmk:local_lamination}
 Under the assumptions of Proposition \ref{prop:lamination}, there exists a neighborhood $\mathcal N$ of $\jstar$ and a lamination $\cW^u$ of $
 \mathcal N$ by Riemann surfaces which extends the family of local unstable manifolds of saddle points. Beware however that it does 
 not \emph{a priori} imply that $J^-\cap \mathcal N$ is laminated nor that it coincides with $\supp(\cW^u)$: indeed, $J^-$ is the closure of \emph{global} unstable manifolds, which could  recur to $\mathcal N$ in a complicated fashion 
 (this point  is a main issue in \cite{saddle}).
 \end{rmk}

\subsection{Existence of invariant laminations and hyperbolicity}
  
 Recall that a complex Hénon map 
 $f$ is said to be 
  {\em hyperbolic} if $J$ is a hyperbolic set. As was noted above, by \cite{saddle} (see also \cite{guerini peters})
 it is actually enough to check hyperbolicity on $\jstar$: this opens the way to hyperbolicity criteria based on 
 periodic points. 

\begin{thm}[{\cite{saddle}}] \label{thm:saddle}
If $\jstar$ is a hyperbolic set for $f$, then $f$ is hyperbolic. 
\end{thm}

A geometric criterion for hyperbolicity based on the existence and  transversality of unstable laminations 
was established  in  \cite[Thm 8.3]{bs8}.  By  incorporating 
  the result of Theorem \ref{thm:saddle} it reads as follows.
  
 \begin{thm}[{\cite{bs8}}]   \label{thm:bs8}
 Let $f$ be a complex Hénon map. Assume that there exists a neighborhood of $\jstar$ and   Riemann surface laminations $\el^\pm$ of $J^\pm$ such that $\el^+$ and $\el^-$ intersect transversally at all points of $\jstar$. Then $f$ is hyperbolic. 
 \end{thm}

It is convenient to formulate this result in the language of uniform regularity. The following is 
 an essentially equivalent statement (see however Remark \ref{rmk:local_lamination}). 
 
 \begin{thm}  \label{thm:bs8_revisited}
 Let $f$ be a complex Hénon map. If every point in $\jstar$ is uniformly regular and transverse then $f$ is hyperbolic.
 \end{thm}
 
Let us give a self-contained proof of this theorem, which
   basically follows  the approach of \cite[Thm 8.3]{bs8}.  
First, recall from Proposition \ref{prop:lamination} 
 that if every point in $\jstar$ is uniformly u-regular, then exists $r>0$ and 
  a lamination $\cW^u$ in the $r$-neighborhood of $\jstar$,
 extending the unstable manifolds of saddle points. 
 Recall also   the  dynamical Green function $G^+$, defined by 
 $G^+(x)= \lim_{n\to\infty} d^{-n} \log^+ \norm{f^n(x)}$. It is a non-negative   continuous psh function in $\cd$, with 
 the property that $\set{G^+ = 0}   = K^+$.

 \begin{prop} \label{prop:uuniform}
 Let $f$ be a complex Hénon map. Assume that  every point in $\jstar$ is  uniformly u-regular and 
 that   for every $x\in \jstar$,   
  $ G^+\rest{\cW^u_\loc(x)}\not\equiv 0$.   
Then  $f$ is uniformly expanding  in the direction of   $T\cW^u$ along $\jstar$.
\end{prop}

This condition on $G^+$ will be used several times in the sequel; it means that $G^+$ does not vanish identicaly on any neighborhood of $x$ in $\cW^s(x)$.


\begin{proof}
Let $r$ be such that for any saddle point $p$, $W^u(p)$ has size $5r$ at $p$. Then by Proposition  \ref{prop:lamination}, 
$\cW^u$  defines  a lamination   in the $r$-neighborhood of $\jstar$ 
     such that for every saddle point $p$, $\cW^{u}(p)$ 
 coincides with the local unstable manifold of $p$. 

We have to show that $f$ is   uniformly expanding along   $\cW^u\rest{\jstar}$, that is,  
 there exists $C>0$ and $\lambda>1$ such that for every $x\in \jstar$, every $k\geq 1$ 
  and $e\in T_x\cW^u(x)$, 
$  \abs{Df^k_x(e)}\geq  C \lambda^k \abs{e}$
(where $\abs{\cdot}$ denotes  the   Riemannian metric induced by the standard Hermitian structure of $\cd$).
 By continuity it is enough to prove this property on  the (dense) set $\mathcal S$ of saddle periodic points. 
For this, we will construct a 
 metric $\abs{\cdot}'$ on $T\cW^u\rest{\mathcal S}$ 
which is   equivalent to  the ambient one (with uniform constants) 
and such that    for every $p\in \mathcal S$, 
  and $e\in T_p\cW^u(p)$, $\abs{Df_p(e)}'\geq   \lambda  \abs{e}'$.

For every saddle point $p$, the global unstable manifold is biholomorphic to $\cc$, so its uniformisation 
$\psi^u_p:\cc\to W^u(p)$ is unique up to a multiplicative factor at the source. In particular 
 $f$ is affine is these parameterizations, and
there is a well-defined notion 
of a round disk in $W^u(p)$, which is $f$-invariant.   
For $e\in T_pW^u(p)$ and $\eta>0$ we define $\abs{e}_\eta$ in the style of the Kobayashi metric:
$$\abs{e}_\eta = \unsur{\rho_\eta} \text{ where } 
\rho_\eta = \sup\set {\rho , \    \sup_{D(0, \rho)} G^+\circ \psi^u_p  \leq  \eta \text{ where } \psi^u_p:\cc\tilde{\to} W^u(p) \text{ and } 
(\psi^u_p)' (0) = e }$$

For every $x\in \jstar$ (not necessarily a saddle),  
$G^+\rest{\cW^u_\loc(x)}$ is not identically 0  near $x$ so we infer that for every $r>0$,  $\sup G^+\rest{\cW^u_r(x)}>0$. 
From the continuity of the Green function, the compactness of $\jstar$,  and the lamination structure we infer the  
existence of     constants   $r>0$ and  
$\eta_1' > \eta_1>0$ such that for every  $x\in \jstar$, $$ \eta_1 \leq  \sup G^+\rest{\cW^u_r(x)}\leq \eta_1'.$$

%

Recall that for near any $x\in \jstar$, up to a unitary change of coordinates, $\cW^u$
 is a union of graphs over a disk of size $r$ and slope bounded by 1 (relative to some projection $\pi$). Thus if $p\in \mathcal S$ is close to $x$, and $\psi_p^u$ is as above, 
it follows that   $\pi\circ \psi^u_p\rest{D(0, \rho_{\eta_1})}$ is  a univalent holomorphic function. Set $\eta_2 = \eta_1/2$.
The Koebe distortion theorem together with the uniform continuity of the Green function  imply that for 
$\rho<\rho_{\eta_2}$,  $\pi\circ \psi^u_p(D(0, \rho))$  is approximately a round disk (with uniform distortion bounds).  
From this uniformity, we infer that there exists
$\la>1$ such that for every $p\in \mathcal S$ and $e\in T_p\cW^u(p)$, 
 $\abs{e}_{\eta_2/d} \geq \lambda  \abs{e}_{\eta_2}$. Set $\eta_3 = \eta_2/d$.
 The invariance relation of the Green function $G^+$ implies that $\abs{df_p(e)}_{d\eta_3} = \abs{e}_{\eta_3} $. From this 
   we get that  for every $p\in \mathcal S$ and $e\in T_p\cW^u(p)$
 $\abs{df_p(e)}_{\eta_2} \geq \lambda \abs{e}_{\eta_2}$.  Finally, again from the uniform continuity of the Green function and 
 bounded distortion, we get that $\abs{\cdot}_{\eta_2}$ is (uniformly) equivalent to $\abs{\cdot}$ on $T\cW^u\rest{\mathcal S}$ so 
the proof is complete. 
 \end{proof}
  
  \begin{rmk}\label{rmk:continuous}
  By a standard procedure, up to reducing $\lambda$
   it is possible to construct a {\em continuous} metric $\abs{\cdot}''$   on $T\cW^u\rest{\jstar}$ such that  
 $\abs{Df_x(\cdot)}''\geq   \lambda  \abs{\cdot}''$. Indeed for $\e>0$ and $e\in T_x \cW^u(x)$, put 
 $$\abs{e}'' = \sum_{n=0}^\infty (\lambda - \e)^n \abs{Df^{-n}_x (e)}. $$ Then one easily checks that $\abs{\cdot}''$ is well-defined, 
 continuous, and satisfies 
 $$\abs{Df_x\inv (e)}''\leq 
 (\lambda-\e)\inv \abs{e}''.$$ 
  \end{rmk}

 The next result implies that if $f$ is uniformly regular, then the second assumption of Proposition \ref{prop:uuniform} holds.
  
  \begin{prop}[see {\cite[Prop. 4.7]{hyperbolic}}]\label{prop:regularG}
 Let $f$ be a complex Hénon map. If  $x\in \jstar$ is    regular 
then   $  G^+\rest{\cW^u_\loc(x)}\not\equiv 0$.
  \end{prop}
  
 \begin{proof}
 If $(p_n)$ is a sequence of distinct saddle points converging to $x$, then $W^s_r(p_n)$ is a sequence of 
 disjoint  submanifolds converging to  
  $\cW^s_r(x)$. Since by assumption $\cW_r^u(x)$ and $\cW^s_r(x)$ are distinct, then for large $n$
$W^s_r(p_n)$ must possess  transverse intersection points with $\cW^u_r(x)$ close to $x$: 
if  $\cW_r^u(x)$ and $\cW^s_r(x)$ 
are transverse this is clear, and if they are tangent this follows from 
 \cite[Lemma 6.4]{bls}).  Then      the inclination lemma implies that $(f^n\rest{\cW^u_\loc(x)} )$
  is not a normal family of holomorphic mappings, therefore $G^+$ is not harmonic on  $\cW^u_\loc(x)$, thus not 
  identically zero, and we are done. 
 \end{proof}
 
 \begin{proof}[Proof of Theorems \ref{thm:bs8} and \ref{thm:bs8_revisited}]
 Under the assumptions of Theorem \ref{thm:bs8_revisited}, it follows directly from Propositions
  \ref{prop:uuniform} and \ref{prop:regularG} that $f$ is uniformly expanding   along $T\cW^u\rest{\jstar}$ and 
  contracting along  $T\cW^s\rest{\jstar}$, that is,  $\jstar$ is a hyperbolic set. Then we conclude from
   Theorem \ref{thm:saddle} that $f$ is hyperbolic. 
   
 To establish Theorem \ref{thm:bs8}, it is enough to check that the existence of the transverse 
  laminations $\el^+$ and $\el^-$ imply uniform 
 transverse regularity. 
We first  observe   that for any saddle point $p$, $W^s(p)$ locally coincides with the leaf $\el^+(p)$ of $\el^+$ through $p$, 
and likewise 
 in the unstable direction. Indeed 
 since the  leaves of $\el^+$ are contained in $J^+$,   for every disk $\Delta$ contained in such a leaf, $(f^n\rest\Delta)_{n\geq 0}$ is a normal family. Now if $\el^+_\loc(p)\neq W^s_\loc(p)$ then either they are transverse and it follows from the inclination lemma that 
 $(f^n\rest{\el^+_\loc(p)})_{n\geq 0}$ is not normal. Otherwise by \cite[Lemma 6.4]{bls} for any $x\in \jstar $ close to $p$, 
$ \el^+_\loc(x)$ is transverse to $W^s_\loc(p)$ and similarly  $(f^n\rest{\el^+_\loc(x)})_{n\geq 0}$ is not a normal family. 
In both cases we reach a contradiction. 
It then follows from Proposition \ref{prop:lamination} that 
 every point in $\jstar$ is uniformly regular and transverse and we conclude as before. 
\end{proof}
 
It turns out that the transversality assumption in Theorem \ref{thm:bs8} is unnecessary, that is, uniform regularity 
 rules out the possibility of tangencies. 

  \begin{thm}\label{thm:uniform regularity}
  Let $f$ be a complex Hénon map. If every point in $\jstar$ is  uniformly  regular then 
 $f$ is uniformly hyperbolic.
  \end{thm}
     
\begin{proof}
By Proposition  \ref{prop:lamination} there exist laminations $\cW^u$ and $\cW^s$ in a neighborhood 
of $\jstar$ extending the family of  local stable and unstable manifolds of periodic points, and 
by Proposition \ref{prop:uuniform} and \ref{prop:regularG} we get 
 that $f$ is uniformly expanding along $\cW^u$ and  $f\inv$ is uniformly
expanding along $\cW^s$. 
To prove the  theorem we thus have to show that these laminations  are transverse at all points of $\jstar$. 
 Let $\mathcal{T}$ be the tangency locus, that 
is the set of points $x\in \jstar$ such that $\cW^s(x)$ and $\cW^u(x)$ are tangent at $x$. This is a closed invariant set. Assume
by way of contradiction  that it  is non empty. Then it supports an ergodic invariant measure $\nu$. Let 
$$\chi^+ = \lim_{n\to +\infty} \unsur{n} \int \log\norm{df^n_x} d\nu(x) \text{ and } 
\chi^- = \lim_{n\to +\infty} \unsur{n} \int \log\norm{df^{-n}_x}\inv d\nu(x) $$ be the Lyapunov exponents of $\nu$. Since $f$ is uniformly expanding/contracting  along  $\cW^{u/s}$ we infer that $\chi^-<0<\chi^+$. 
By Oseledets' theorem, there exists an associated  invariant 
measurable  decomposition $T_x\cd   = E^-(x)\oplus E^+(x)$ defined $\nu$-a.e. such that the growth rate 
of vectors in $E^{\pm}(x)$ 
is governed by $\chi^\pm$. 
By Pesin's theory (see e.g. \cite{fhy}) for $\nu$-a.e. $x$ there are local  
stable and unstable manifolds $W^s_{\mathrm{Pesin}}(x)$ and $W^u_{\mathrm{Pesin}}(x)$
respectively 
tangent to the characteristic directions associated to the negative and positive exponent.
But by Lemma \ref{lem:pesin_regular},
 $W^{s/u}_{\mathrm{Pesin}}(x)$ locally coincides with $\cW^{s/u}(x)$,
so we infer  that $E^+(x) = E^-(x)$ a.e. which contradicts the Oseledets theorem. This contradiction finishes the proof.
\end{proof}

If $f$ is not volume preserving we  can further relax the previous criterion. 
  
  \begin{thm}\label{thm:us regularity}
  Let $f$ be a complex Hénon map with $\abs{\jac(f)}\neq 1$. If every point in $\jstar$ is uniformly 
   u- and s-regular then   $f$ is uniformly hyperbolic.
  \end{thm}
     
\begin{proof}
The difference with Theorem \ref{thm:uniform regularity} is that  $\mathcal {T}$ can now contain local leaves so 
  Proposition \ref{prop:uuniform} does not apply. Let $r$ is the uniform size of local s/u manifolds along $\jstar$.
Without loss of generality assume that $\abs{\jac(f)}< 1$. Note
that if $\cW_\loc^u(x) = \cW_\loc^s(x)$ then $\cW^u_r(x) = \cW^s_r(x)$.  
Denote by $\mathcal{T}'$ the set of points $x\in \jstar$ such that $\cW^u_r(x) = \cW^s_r(x)$.  Then $\mathcal  T'$ is also closed and invariant. Indeed if  $\cW^u_r(x) = \cW^s_r(x)$ then clearly $\cW_\loc^u(f(x)) = \cW_\loc^s(f(x))$, hence 
$\cW_r^u(f(x)) = \cW_r^s(f(x))$. Thus 
$f(\mathcal  T')\subset \mathcal  T'$, and the closedness of $\mathcal{T}'$ follows directly 
from the continuity of $x\mapsto \cW^{u/s}_r(x)$. 
 
Assume by way of
contradiction that   $\mathcal{T}'$ is non-empty.  Then it supports an ergodic  invariant measure $\nu$. 
Since $f$ is dissipative its Lyapunov 
exponents satisfy $\chi^-<0\leq\chi^+$.  For every $x\in \mathcal T'$, $\cW_r^u(x) = \cW_r^s(x)$ 
is contained in $J^+\cap J^-$ so it is 
a Fatou disk under forward and backward iteration. The following lemma relates these disks to the Oseledets decomposition.

\begin{lem}\label{lem:Oseledets}
Let $f$ be a complex Hénon map and  $\nu$ be an ergodic  invariant measure whose Lyapunov 
exponents satisfy $\chi^-<0\leq \chi^+$, 
 and  $T_x\cd  = E^-(x)\oplus E^+(x)$ be  the associated 
  measurable   decomposition. 
  If $\nu$-a.e. point is u-regular then for $\nu$-a.e. $x$, $\cW_\loc^u(x)$ is tangent to $E^+(x)$ at $x$. 
\end{lem}

Assuming this result for the moment, let us conclude the proof. 
The contradiction hypothesis implies that  for $\nu$-a.e. $x$,  $\cW_\loc^u(x) = \cW^s_\loc(x)$. 
By Pesin's theory a $\nu$-generic point $x$  admits a local strong stable manifold $W^s_{\rm Pesin}(x)$, 
which is tangent to $E^-(x)$, and by Lemma \ref{lem:pesin_regular} it
coincides with $\cW^s_\loc(x)$. 
 On the other hand, by Lemma \ref{lem:Oseledets},  $\cW^u_\loc(x)$ is a.s. transverse  to $E^-(x)$.
  This contradiction shows that 
  $\mathcal{T}'$ is empty. Therefore every point in $\jstar$ is regular and applying 
  Theorem \ref{thm:uniform regularity} finishes the proof. 
\end{proof}

\begin{proof}[Proof of Lemma \ref{lem:Oseledets}]
Since $\cW^u_\loc(x)$ is contained in $J^-$, $\lrpar{f^{-n}\rest{\cW^u_\loc(x)}}$ is a normal family, so it follows from the Cauchy estimates 
that $\norm{df^{-n}_x (e^u(x))}$ is bounded, where  $e^u(x)$ is any tangent vector to $\cW^u(x)$ at $x$. On the other hand 
the Oseledets theorem  asserts that almost surely, if $e(x)$ is any non-zero vector such that $e(x)\notin E^+(x)$, 
$\norm{df^{-n}_x (e(x))}$ grows exponentially at rate $\abs{\chi^-}$. Hence $e^u(x)\in E^+(x)$ and we are done. 
\end{proof}

\subsection{Unstable lamination, dominated splitting and hyperbolicity}
It is natural to expect that  in the dissipative setting, uniform u-regularity is enough to characterize hyperbolicity.
Indeed, uniform u-regularity should provide 
uniform expansion along some field of    directions, which,  together with volume contraction  yields uniform hyperbolicity. 
The basic technical tool needed to implement this idea is that of  \emph{dominated splitting}.
Recall that a {dominated splitting} on some invariant set $\Lambda$ is a splitting 
of the form 
$T\cd\rest{\Lambda} = E^s\oplus E^c$ for which there exists $C>0$ and $\la<1$ such that 
$$\frac{\norm{df^n\rest{E^s}}}{\norm{df^n\rest{E^c}}} \leq C\lambda^n. $$ Then this splitting is automatically continuous, and if 
$\abs{\jac(f)}\leq 1$ the direction $E^s$ is contracting. The existence of a dominated splitting for $f$ 
along $J$ is a way to formalize the ``absence of critical points'' on $J$. 

Our first  result can be viewed as a     version of \cite{lyubich peters} in a (greatly) simplified setting. 

\begin{prop}\label{prop:dominated}
 Let $f$ be a complex Hénon map with $\abs{\jac(f)}\leq 1$. If every point in $\jstar$ is uniformly u-regular and if $f$ admits a 
 dominated splitting on $\jstar$, then $f$ is hyperbolic. 
\end{prop}

\begin{proof}
Dominated splitting implies the existence of a strong stable lamination $\cW^s$
 in a neighborhood of $\jstar$, hence points of $\jstar$ are 
uniformly s-regular. Then if $\abs{\jac(f)}<1$, the result  follows directly  from Theorem \ref{thm:us regularity}. In the general case 
we    just have to repeat the proof of Theorem \ref{thm:us regularity}, the only difference being
 that dissipativity was used there to show 
that $\nu$ has a negative exponent while here this follows from the dominated splitting assumption.  
\end{proof}

The  idea of dominated splitting  shows that hyperbolicity   already holds 
under the assumptions of Proposition \ref{prop:uuniform}:

\begin{prop}\label{prop:uuregular}
Let $f$ be a complex Hénon map with $\abs{\jac(f)}\leq  1$. If every point in  $\jstar$ is uniformly u-regular and 
for every $x\in \jstar$, $G^+\rest{\cW^u_\loc(x)}\not \equiv 0$ then $f$ is hyperbolic. 
\end{prop}

Applying  Proposition \ref{prop:regularG} yields the following corollary, which generalizes (and gives a new approach to) 
Theorem \ref{thm:uniform regularity}. 

\begin{cor}
Let $f$ be a complex Hénon map with $\abs{\jac(f)}\leq  1$.  
If every point in  $\jstar$ is regular and uniformly u-regular then $f$ is hyperbolic. 
\end{cor}

\begin{proof}[Proof of Proposition \ref{prop:uuregular}]
By the cone criterion for dominated splitting  (see  \cite[Prop. 2.2]{sambarino}) it is enough to 
 prove that for every $x\in \jstar$ there exists a cone $\mathcal{C}_x$ about $T_x \cW^u(x)$ 
 in $T_x\cd$ such that  the field of  cones   $(\mathcal{C}_x)_{x\in \jstar}$  is strictly contracted by the dynamics. Then the 
 result follows from Proposition \ref{prop:dominated}. 
By  Proposition \ref{prop:uuniform}   and Remark \ref{rmk:continuous}
there is a continuous Riemannian metric on $T\cW^u\rest{\jstar}$ which is immediately expanded by the dynamics. 
Let 
$(e_x)_{x\in \jstar}$ be a field of tangent vectors to $\cW^u$   of unit norm relative to this metric, and 
$f_x$ be orthogonal to $e_x$ in $T_x\cd$ (relative to the ambient Riemannian structure) and such that 
$\det(e_x, f_x) = 1$. For small $\e$, define a continuous  field of cones  $\mathcal{C}_x^\e\subset T_x\cd$   by 
$$\mathcal{C}_x^\e = \set{ue_x+vf_x, \ \abs{v} \leq  \e \abs{u}}.$$ 
Working in the frame $\set{(e_x, f_x), x\in   \jstar}$, 
the matrix expression of $df_x$  is of the form 
$$\begin{pmatrix} 
\lambda_x & a(x) \\ 0 & \lambda_x\inv J
\end{pmatrix},$$
where $\abs{\lambda_x}\geq \lambda _0>1$ and 
  $J$ is the Jacobian, so $\abs{J}\leq1$. Since the frame $(e_x, f_x)$ is continuous, $a(\cdot)$ is bounded. Then one checks easily that if $\e$ is so small that $\la_0 -\e\norm{a}>1$, then $$df_x (  \mathcal{C}_x^\e) \subset  \mathcal{C}_{f(x)}^{\lambda_0\inv\e}.$$
 Hence the field of cones $(C_x)_{x\in \jstar}$ is strictly contracted  by the tangent dynamics and we are done (note that a similar 
 argument appears in \cite{closing}). 
\end{proof}

The next result shows that uniform expansion can indeed be deduced   
from the geometric property of uniform u-regularity. Assume that every $x\in \jstar$ is uniformly regular of size $4r$. 
Recalling the construction of global unstable manifolds from local ones, for $x\in \jstar$ we define  
\begin{equation}\label{eq:global_manifold}
\cW^u(x) = \bigcup_{n\geq 0} f^n(\cW^u_r(f^{-n}(x))).
\end{equation}
It follows from this definition that 
 $f^{-1}(\cW^u(x))=f^{-1}(\cW^u_r(x))\cup \cW^u(f^{-1}(x))$, 
 hence  $f^{-1}(\cW^u(x))$  contains $ \cW^u(f^{-1}(x))$ and 
 it is not a priori clear that the $\cW^u(x)$ define an invariant family of curves. 
  However,  if  $\cW^u(x)$ is biholomorphic to $\cc$ for every $x\in \jstar$  
 then $f^{-1}(\cW^u(x))= \cW^u(f^{-1}(x))$, for otherwise $\cW^u(f^{-1}(x))$ would strictly contain 
 $f^{-1}(\cW^u(x))$, and it would be 
 a complex submanifold of $\cd$ biholomorphic to the Riemann sphere, which is  contradictory. 
 The following theorem confirms the expectation that the 
 parabolicity of leaves in $J^-$ is  associated with expansion (compare e.g. \cite[\S 4]{lyubich minsky}).

%
  \begin{thm}\label{thm:uuregular global}
  Let $f$ be a dissipative complex Hénon map. 
  If every point    $x\in \jstar$ is uniformly u-regular and in addition 
  $\cW^u(x)$  is biholomorphic to 
$\cc$, then $f$ is hyperbolic. 
  \end{thm}


Remark that the definition of $\cW^u(x)$ in \eqref{eq:global_manifold} a priori depends on $r$.
The   theorem  shows that  if these manifolds are biholomorphic to 
$\cc$, this is actually not the case.

\begin{proof}
For every $x\in \jstar$, fix a uniformization $\psi^u_x: \cc\tilde{\to} \cW^u(x)$ such that $\psi^u_x(0)=x$, 
 which is normalized by $\abs{(\psi^u)'(0)} = 1$.  
For $\eta>0$, define $R_\eta(x)$ to be the maximal radius of a round disk in $\cc$ such 
that $G^+ \circ \psi^u_x\rest{D(0, R_\eta(x))} \leq \eta$ (this is similar but not identical to the definition of 
$\rho_\eta$ in  Proposition \ref{prop:uuniform}). Since $\cW^u(x)$ is an entire curve contained in $J^-$, $G^+\rest{\cW^u(x)}$ is unbounded so $R_\eta(x)$ is finite. We claim that for every $\eta>0$, there exists $C_\eta>0$ such that 
\begin{equation} \label{eq:uniform_R}
\text{for every }x\in \jstar, \ C_\eta\inv\leq  R_\eta(x) \leq C_\eta. 
\end{equation}
 
Indeed, fix $x\in \jstar$ and let us show that $R_\eta$ is locally uniformly bounded from above and below in a neighborhood of $x$. 
Then by compactness these bounds will be  uniform   on $\jstar$. 
Viewed in the unstable parameterizations $f$ is affine so it maps  circles to  circles. Let  $\Delta^u(x, R) = \psi^u_x(D(0, R))$. We first
claim that  there exists   $k\geq 0$ such that $f^{-k}(\Delta^u(x, R_\eta(x))$ is contained in   $\cW^u_r(f^{-k}(x))$. 
Indeed by definition of $\cW^u(x)$, for every $x'\in \fr \Delta^u(x, 4R_\eta(x))$, there exists $k\geq 0$ such that 
$f^{-k} (x')\in \cW^u_r(f^{-k}(x))$.  As in the proof of Proposition \ref{prop:uuniform},  
the Koebe distortion theorem implies that there is a coordinate $\pi: \cW^u_{4r}(f^{-k}(x))\to \cc$ such that 
 that for $s\leq r$, 
$\pi(\Delta^u(f^{-k}(x), s))$ is approximately a disk of radius $s$. 
Now  $f^{-k} \lrpar{\Delta^u(x, R_\eta(x))}$ is a round disk in the  affine coordinate, and it 
  possesses a boundary point in 
 $\cW^u_{r/2}(f^{-k}(x))$,  so it follows that it is completely contained in $\cW^u_r(f^{-k}(x))$.
Therefore, replacing $x$ by $f^{-k}(x)$ 
we can assume that $\psi^u_x(D(0, R_\eta))\subset \cW^u_r(x)$. 

  By uniform u-regularity, for   
$y$ close to $x$, $ \cW^u_{r}(y)$ is a graph of slope at most 1  over a disk of size $r$ relative to the projection $\pi$. 
Thus from Koebe distortion again, we infer that for $y$ close to $x$, 
  the distance induced by the normalized affine structure along    the $\cW^u_r(y)$
is      equivalent to the ambient distance. In particular there exists a 
constant $K$ depending only on $r$ such that for $y$ close to $x$ and $\eta$ as above, 
$$ K\inv \dist(y, \set{G^+ = \eta}) \leq R_\eta(y) \leq K \dist(y, \set{G^+ = \eta}).$$  Finally by the 
Hölder continuity of $G^+$
$\dist(y, \set{G^+ = \eta})$ is bounded from below by $C\eta^\theta$, 
and if $\dist(y,x)\leq r$ it is bounded from above by $Cr$. This completes the proof of \eqref{eq:uniform_R}. 

Then, from the invariance relation of $G^+$
we have   
  $$f(\Delta^u(x, R_\eta(x)))   = \Delta^u(f(x), R_{d\eta}(f(x)))\supset \Delta^u(f(x), R_{\eta}(f(x))), $$ hence 
  for every $n\geq 0$ we infer that 
 $f^n(\Delta^u(x, R_\eta(x))) \supset \Delta^u(f^n(x), R_{\eta}(f^n(x)))$.
In particular for every $x\in \jstar$  and every $n\geq 1 $ we have that 
that $f^n(\Delta^u(x, C_\eta))\supset \Delta^u(f^n(x), C_\eta\inv)$.  Again since $f$ is affine in the unstable parameterizations we deduce that for every $t>0$, $$f^n(\Delta^u(x, tC_\eta))\supset \Delta^u(f^n(x), tC_\eta\inv).$$
Finally taking the derivative at $t=0$ we conclude that $\norm{Df^n_x\rest{T_x\cW^u}}\geq (C_\eta)^{-2}$. 

This bound in turns implies the existence of a dominated splitting along $\jstar$. 
This follows from the criterion of Bochi-Gourmelon \cite[Thm A]{bochi gourmelon} (see also Yoccoz \cite{yoccoz}). Indeed since 
$f$ has constant Jacobian,  for $x\in \jstar$   
 the singular values of $Df^n_x$ are $\sigma_n^+$ and $\sigma_n^-= J^n/\sigma_n^+$, where $J    = \abs{\jac(f)}<1$, 
 and $\sigma_n^+ \geq (C_\eta)^{-2}$.
 Therefore $$\frac{\sigma_n^+}{\sigma_n^-}  = \frac{(\sigma_n^+)^2}{J^n}\geq   {\frac{1}{C_\eta^4 J^n}}$$ so 
 \cite{bochi gourmelon} applies and we get a dominated splitting on $\jstar$. Applying proposition \ref{prop:dominated} concludes the proof. 
\end{proof}

\begin{rmk}
If $J^-$ is  globally laminated (outside a finite set  of periodic points, say) one might expect that 
the additional assumption that $\cW^u(x)\simeq \cc$ for every $x$  in Theorem \ref{thm:uuregular global}
  would  follow  from the density of unstable manifolds of saddle points. Unfortunately there are examples of 
minimal Riemann surface laminations 
containing both parabolic and hyperbolic leaves (see \cite[Thm 6.6]{ghys_survey})
\end{rmk}

\subsection{Concluding remarks}
\subsubsection{}
Uniform s-regularity on $\jstar$ does not imply hyperbolicity. Indeed there 
 are examples of   Hénon mappings  with parabolic points and a dominated splitting on $\jstar$ 
 (see \cite{radu tanase, lyubich peters}).  It would be interesting to know whether uniform s-regularity on $\jstar$ implies the 
 existence of a dominated splitting.

\subsubsection{}
The only property of $\jstar$ that was used  in the various hyperbolicity criteria in this section is that 
$\jstar$ is a closed invariant set in which saddle periodic points are dense. 
So in all these results we could with  an arbitrary  closed invariant set $\Lambda$, in which saddle points are dense. The 
notion of   uniform u-regularity  has to be replaced 
 by uniform u-regularity along $\Lambda$, meaning that the uniform 
 size of unstable manifolds holds only   for  sequences of saddle points in $\Lambda$, and likewise for s-regularity. 
Then there are statements analogous  to Theorems 
\ref{thm:bs8_revisited}, \ref{thm:uniform regularity}, \ref{thm:us regularity} and \ref{thm:uuregular global}, in which uniform regularity 
is replaced by uniform regularity along $\Lambda$,  and the conclusion 
is that $\Lambda$ is a hyperbolic set.

\section{A topological criterion for   hyperbolicity}\label{sec:topological}

In this section we work in the setting of Conjecture \ref{conj:topological}:  We assume that  
 $f_0$ and $f_1$ are two complex Hénon maps such   that $f_0$ is hyperbolic, and 
   that there exist  respective  neighborhoods $N_0$ and $N_1$ 
 of $J_0=J_0^\varstar$ and $J_1^\varstar$ and a conjugating homeomorphism $\phi: N_0\to N_1$. 
 Our purpose is to show that $f_1$ is hyperbolic on $J_1^\varstar$. 
\subsection{Periodic points and their (un)stable manifolds.}

\begin{prop}\label{prop:periodic}
Let $f_0$ and $f_1$ be as in Conjecture \ref{conj:topological}. Then $\phi(J_0) = J_1^\varstar$. 
If $f_1$ is dissipative then  all periodic points  of $f_1$ on $J_1^\varstar$ are saddles. If $f_1$ is conservative the same holds provided $\phi$ is   H\"older continuous. 
\end{prop}

\begin{proof}
The first assertion is a direct consequence of the equidistribution of periodic orbits. Indeed the topological conjugacy 
shows that $f_0$ and $f_1$ have the same entropy, hence the same dynamical degree. Since periodic orbits equidistribute towards the maximal entropy measure, we get that $\phi_\varstar\mu_0 = \mu_1$. Since $\supp(\mu_1) = J_1^\varstar$, we infer that $\phi(J_0)  = J_1^\varstar$.  (On the other hand it is unclear at this stage whether $J_1 ^\varstar = J_1$.) 

 Any periodic point on $J_1^\varstar$ admits 
a neighborhood in which it is topologically conjugate to a saddle. Let 
$p\in J_1^\varstar$ be some periodic point which we may suppose fixed. Assume that $f_1$ is dissipative.
 Then if $p$ is not a saddle
  it is semi-attracting. By the hedgehog theory of \cite{FLRT, LRT} there exists 
in some neighborhood of $p$ a non-trivial totally invariant set $\mathcal{H}$ 
made of points which do not converge to $p$ under backward nor forward iteration: indeed there is a subsequence 
$q_n$ such that $f^{q_n}\to \mathrm{id}$ on $\mathcal{H}$. 
This is not compatible with the local conjugacy to a saddle fixed point, therefore we conclude that $p$ is a saddle.

If $f$ is conservative and $p$ is not a saddle, then it is neutral.  
Since $\phi$ is Hölder, then points in 
$\phi(W^s_\loc (\phi\inv(p)))$ converge to the origin exponentially fast. On  the other hand for every $\e>0$
there exists a norm on $\cd$ for which 
$$(1-\e)\norm{v}\leq \norm{df_p(v)} \leq (1+\e)\norm{v}.$$ 
Indeed if $df_p$ is diagonalizable this is clear since the eigenvalues have modulus 1, 
and otherwise we can make $df_p$ triangular with the off-diagonal term as small as we wish, and take an adapted 
norm. Then if $x$ 
is close to $p$ and $f_1^n(x) \to p $ we infer that $\norm{f^n(x) - f^n(p)}\geq (1-2\e)^n \norm{x-p}$ which is contradictory 
if $\e$ is small enough.   Thus again we conclude that all periodic points on $\jstar_1$ are saddles. 
\end{proof}


For a saddle point $p$, we now denote by $W^s_r(p)$   the component of $W^s(p)\cap B(p,r)$ containing $p$.
We fix $r_0$ such that  for every $x\in J_0$, $W^s_{r_0}(x)$ (resp. $W^u_{r_0}(x)$)
is a properly embedded holomorphic disk with the property that there exist uniform 
$C>0$ and $0<\lambda<1$ such that for every $x'\in W^s_{r_0}(x)$ (resp. $W^u_{r_0}(x)$),  
$\dist(f^n(x), f^n(x'))\leq C\lambda^{\abs{n}}$ when $n\to \infty$ (resp  $n\to -\infty$). 
We also assume for further reference that  $f$ has   product structure 
 in  the $2r_0$-neighborhood of $\jstar_0$. 

\begin{prop}\label{prop:uniform}
Let $f_0$ and $f_1$ be as in Conjecture \ref{conj:topological}. There exists $r_1>0$ such that for any 
saddle periodic point 
$p$ for $f_1$, $W^s_{r_1}(p)$ (resp. $W^u_{r_1}(p)$) is a submanifold of $B(p,r_1)$. 
\end{prop}

 \begin{proof}
Without loss of generality we treat the case of stable manifolds and assume that $p$ is fixed. 
 If $r_0$ is as above there exists $r_1>0$ such that for every $y\in \jstar_1$, 
$B(y,r_1)\Subset \phi(B(\phi\inv(y), r_0))$. We claim that for every saddle fixed  point $p$ for $f_1$, 
\begin{equation}\label{eq:inclusion}
W^s_{r_1}(p) \subset \phi\lrpar{W^s_{r_0}(\phi\inv(p))} \subset W^s(p).
\end{equation}
The right inclusion is obvious since belonging to 
$W^s(p)$ is characterized by the 
 topological   property that $f^{n}(y)\to p$. For the left inclusion, just observe that 
$\phi\inv\lrpar{W^s_{r_1}(p)}$ is a connected   subset of  $W^s(\phi\inv(p))\cap B(\phi\inv(p), r_0)$ containing $p$, 
hence it is  contained in $W^s_{r_0}(\phi\inv(p))$. 

To show that it is properly embedded, we first observe that there exists $r=r(p)$ such that 
$W^s_{r} (p)$ is properly embedded in $B(p,r)$. By the invariance of domain theorem $\phi\inv(W^s_{r} (p))$ is a 
neighborhood of $\phi\inv(p)$ in $W^s(\phi\inv(p))$. Thus it follows that there exists $ n =n(p)$ such that 
$f_0^n (W^s_{r_0}(\phi\inv(p))) \subset \phi\inv(W^s_{r} (p))$. 
Then  from \eqref{eq:inclusion}    we get that 
 $ f_1^n(W^s_{r_1}(p))  \subset  W^s_{r} (p) $, so $W^s_{r_1}(p) \subset f_1^{-n} W^s_{r_1}(p)$. From this we conclude that 
 $W^s_{r_1}(p)$ is properly embedded in $B(p,r_1)$, as desired. 
 \end{proof}

\begin{rmk} 
At this stage we know that stable manifolds are properly embedded in a ball of uniform size, but since in the last argument 
the quantities 
$n$ and $r$ are a priori not uniform in $p$, we have no uniformity for the geometry of $W^s_{r_1}(p)$. Obtaining such a uniformity 
will be  the purpose of the forthcoming arguments.
\end{rmk}

\subsection{Tube argument}

\begin{lem}[Tubular neighborhood lemma]\label{lem:tub}
If $\Delta$  is a  subvariety in $B(0, 2r)$ of size $r$ at 0 then  
 there exists $\eta = \eta(r)$ such that if $\Delta'$ is a subvariety in $B(0, 2r)$ such that 
$d_{H}(\Delta, \Delta')<\eta$ in $B(0,2r)$, then $\Delta'$ is a branched cover over $\Delta$ in $B(0, r/2)$. (Here $d_H$ denotes the
Hausdorff distance.)
\end{lem}

\begin{proof}
After  a unitary change of coordinates, $\Delta$ is a graph $y=\psi(x)$
of slope at most 1 in the  bidisk  $ D(0,r)^2$. Since $\psi'(0) = 0$ by the Schwarz lemma 
 we have $\abs{\psi'(x)}\leq \abs{x}/r$ so actually $\abs{\psi(x)}\leq r/2$. It follows that if 
 $\eta<r/4$ and $d_{H}(\Delta, \Delta')<\eta$ in  $ D(0,r)^2$, $\Delta'$ is horizontal in this bidisk. Thus 
  it is a branched cover over the first coordinate, hence over $\Delta$. 
\end{proof}

\begin{figure}[h]
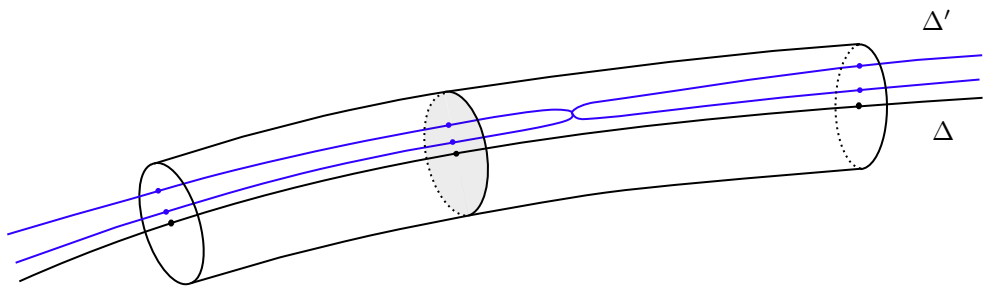
\caption{Tubular neighborhood lemma}
\end{figure}

The ``tube argument'' alluded to in the title consists in applying the previous lemma to construct invariant laminations 
in the setting of the conjecture. Here is a sample statement. 

\begin{prop}\label{prop:saddle regular}
Let $f_0$ and $f_1$ be as in Conjecture \ref{conj:topological}. Then   every saddle periodic point $p$ for $f_1$ 
is uniformly transverse regular. 
\end{prop}

\begin{proof}
Let $p$ be as in the statement of the proposition, and assume that $W^s(p)$ has size $r$ at $p$. 
Reducing $r$ if necessary, we may assume that $\phi\inv(W_r^s(p))$ is contained in a flow box of the stable lamination 
of $f_0$. 
Without loss of generality we may also assume that $r<r_1/4$, where $r_1$ is as in Proposition \ref{prop:uniform}.  We will show that 
there exists a neighborhood $V$ of $p$ such that if $q\in V$ is another periodic point, then $W^s_r(q)$ is a graph over 
$W^s(p)$ in  $B(0, r/2)$.  
This implies that $p$ is uniformly s-regular. Uniform u-regularity is proven in the same way, and the transversality property is 
obvious since $p$ is a saddle. 

We know that for any saddle point $q\in J_1^\varstar$, the  stable manifold $W^s_{r_1}(q)$ is properly embedded in $B(q, r_1)$. In 
addition, by \eqref{eq:inclusion} it is contained in $\phi(W^s_{r_0} (\phi\inv(q)))$. 
By the  uniform continuity of $\phi$ there exists a neighborhood $V$ 
of $p$ such that for every saddle point $q\in V$, $W^s_{r}(q)$ is $\eta$-close to $W^s_{r}(p)$ in $B(p, 2r)$, where $\eta$ is as in the 
tubular neighborhood lemma. Thus $W^s_{r}(q)$ is a branched cover over $W^s_{r}(p)$ in $B(p, r/2)$, and to conclude the 
proof it remains to show that this cover has degree 1. 
By the product structure of $f_0$ in $(\jstar_0)_{2r_0}$ we
have that $W^s_{r_0}(\phi\inv(p))\cap W^u_{r_0}(\phi\inv(p))  = \set{\phi\inv(p)}$ hence 
$W^s_{r_1}(p)\cap  W^u_{r_1}(p) = \set{p}$. Thus,
reducing $\eta$ if necessary, to compute the degree of this branched cover it is enough 
to count the number of intersection points, with multiplicity,  between $W^s_{r}(q)$ and $W^u_{r_1}(p)$. 
Applying  the product structure again we get that $W^s_{r}(q)\cap W^u_{r_1}(p)$ is a single point. Furthermore it is well-known that 
 the order of contact between two smooth complex curves in $\cd$ is a topological invariant. Indeed if we consider 
 two smooth curves $C$ and $D$ with an isolated intersection at $0\in \cd$ 
and intersect them with a small sphere $\mathbb S$ about 0, then $C\cap \mathbb S$ winds $n$ times about 
$D\cap \mathbb S$ where $n$ is the intersection multiplicity. 
 So  we conclude  that the  intersection $W^s_{r}(q)\cap W^u_{r_1}(p)$ 
 is transverse and we are done.
\end{proof}
 
 \subsection{Proof of  Conjecture \ref{conj:topological} in the dissipative case.}

\begin{thm}\label{thm:topological dissipative}
Let $f_0$ and $f_1$ be two polynomial automorphisms of $\cd$ with non-trivial dynamics, and assume that $f_0$ is hyperbolic and 
that $f_1$ is dissipative.

 Suppose that there exists  respective  neighborhoods $N_0$ and $N_1$ 
 of $J_0 = \jstar_0$ and $\jstar_1$ and a homeomorphism $\phi: N_0\to N_1$ such that $\phi\circ f_0 = f_1\circ \phi$ where 
these  compositions makes sense. Then $f_1$ is hyperbolic. 
 \end{thm}

To prove the theorem, let $\cW^s_0$ and $\cW^u_0$ be the   stable and unstable laminations in  $N_0$ and 
$\el^s_1$ and $\el^u_1$ be their respective images under
 $\phi$. At this stage $\el^s_1$ and $\el^u_1$  are topological laminations by 
topological disks in $N_1$.  

Define $\Omega$ to be the set of points $x\in \jstar_1$ such that there exists a neighborhood $V$ of $x$ in $\jstar_1$ 
such that for every $y
\in V$, $\el_1^s(y)$ and $\el_1^u(y)$ are holomorphic and of uniform size in $V$. Note that they must be transverse 
by  the  topological invariance of the order of  contact between smooth curves. 
 By construction, $\Omega$ is 
open in $\jstar_1$ and completely invariant (i.e. $f(\Omega) = \Omega$). 
Proposition \ref{prop:saddle regular} shows that $\Omega$ contains all saddle points. 

The main step of the proof is the following:

\begin{lem} \label{lem:tube Pesin}
Let $f_0$ and $f_1$ be as in Theorem \ref{thm:topological dissipative}. Then    any invariant   
measure supported on 
$\jstar_1$ gives full mass to $\Omega$.  
\end{lem}

Theorem \ref{thm:topological dissipative} follows easily. Indeed, if non-empty,
 the complement of $\Omega$ in $\jstar_1$
 is a closed invariant set hence if it is non-empty it supports an   invariant measure $\nu$. By the lemma, 
 $\nu(\Omega) = 1$ hence the contradiction. Therefore we conclude  that $\Omega   = \jstar_1$, in particular all points in 
 $\jstar_1$ are uniformly   regular, and the result follows from Theorem \ref{thm:bs8_revisited}.  \qed

\begin{proof}[Proof of Lemma   \ref{lem:tube Pesin}]
The method is to adapt the  ``tube argument'' to Pesin stable manifolds. 
Under the assumptions of the theorem, 
let $\nu$ be any invariant measure for $f_1$ supported on $\jstar_1$. Then by Oseledets' Theorem 
for $\nu$-a.e. $x$ there exist Lyapunov exponents $\chi_1(x) \leq \chi_2(x)$ satisfying 
$\chi_1(x)+ \chi_2(x) = \log \abs{\jac(f)}$. In addition since $\nu$ is not concentrated on a periodic orbit we 
have $\chi_2(x)\geq 0$ a.e. hence   $\chi_1(x)<0$ since $\abs{\jac(f)}<1$.
 By the Pesin stable manifold theorem, for $\nu$-a.e. $x$, there 
exists a local stable manifold $W^s_\loc (x)$ which can characterized as the set of points $y$ sufficiently close to $x$ 
such that $\limsup\unsur{n} \log \dist(f_1^n(y), f_1^n(x))<0$. Pick any point $x$ such that $W^s_\loc (x)$ 
exists. We will show that {\em both} $\el_1^s$ and $\el_1^u$ are laminations by Riemann surfaces near $x$. 

Observe first that by hyperbolicity of $f_0$,
 the local stable manifold of $\phi\inv(x)$ is the set of points $z$ near 
$\phi\inv(x)$ such that $\dist(f_0^n(z), f_0^n(\phi\inv(x)))\to 0$ as $n\to +\infty$.
 Hence $\phi\inv(W^s_\loc(x))\subset W^s_\loc(\phi\inv(x))$. 
Since $\phi\inv$ is continuous and injective, by the invariance of domain theorem, $\phi\inv(W^s_\loc(x))$ is 
neighborhood of $\phi\inv(x)$ in $W^s_\loc(\phi\inv(x))$. Thus $W^s_\loc(x)$ coincides with $\el_1^s(x)$ in a neighborhood 
of $x$.

 Let $r$ be so small that $W^s_\loc(x)$ has size $r$ at $x$ and $W^s_\loc(x)=\el_1^s(x)$ in $B(x, 2r)$. 
Then after a unitary change of 
coordinates as in Lemma \ref{lem:tub}, $W^s_\loc(x)$ is a graph of the form $y = \psi(x)$ over $D(0,r)$. Denote this graph by 
$\Delta^s_r$. For small 
$\eta>0$, we define  $$\mathrm{Tub}_\eta =\mathrm{Tub}_\eta( \Delta^s_r)
= \set{(x, y), \ \abs{x}<r, \ \abs{y-\psi(x)}<\eta}.$$  We  say that a submanifold $M$ of
$\mathrm{Tub_\eta}$  (which extends to some 
 neighborhood of  $\overline  {\mathrm{Tub}_\eta}$) is {\em horizontal} if 
$$M\cap \fr \mathrm{Tub}_\eta\subset \set{(x,y)\in\overline{\mathrm{Tub}_\eta}, \  \abs{x}=r}$$ and similarly it is 
{\em vertical}  if $M\cap \fr \mathrm{Tub}_\eta \cap\set{\abs{x}=r} = \emptyset$. As already observed, if $M$ is horizontal it is a branched covering over the first coordinate, and similarly if it is vertical the restriction of $(x,y)
\mapsto y - \psi(x)$  to $M\cap \mathrm{Tub}_\eta$ is a branched covering over $D(0, \eta)$.

 Exactly as in Proposition  \ref{prop:saddle regular}, if $q\in \jstar$ 
is a  saddle point sufficiently close to $x$, $W^s_\loc(q)$ 
is horizontal in $\mathrm{Tub}_\eta$. Now by the transversality of $\cW_0^s$ and 
$\cW_0^u$, there exists a neighborhood $N$  of $\phi\inv(x)$ such that for any $z\in N$,  the distance between 
$\cW^u_0(z)$ and $\phi\inv(\fr \Delta^s_r)$ is bounded from below by a uniform positive constant. 
By continuity of $\phi$, for any $y = \phi(z) \in \phi(N)$ and 
reducing $\eta$ if necessary we get that  $\dist(\fr \Delta^s_r, \el_1^u(y))> 2\eta$. By Proposition \ref{prop:uniform}, if 
$q$ is a saddle  point close to $x$,  $W^u_{r_1}(q) = \el_1^u(q)\cap B(q, r_1)$ 
is a submanifold in $B(q, r_1)$ for a uniform $r_1$ (which we may assume to be large with respect to $r$ and $\eta$). 
So we conclude that it is a submanifold in a neighborhood of 
$\overline  {\mathrm{Tub}_\eta}$, which must be vertical in $\mathrm{Tub}_\eta$.
Thus we have shown that if $q$ is a saddle periodic point sufficiently close to $x$,  the local stable and unstable manifolds of 
$q$ are respectively horizontal and vertical in $\mathrm{Tub}_{\eta}$, with a single transverse intersection point (for transversality 
again we use the topological invariance of the order of contact). Hence both have covering degree 1 respectively over the horizontal and vertical directions in $\mathrm{Tub}_{\eta}$, i.e. they are graphs. 
 Then by the Schwarz Lemma they have uniformly bounded geometry. So we conclude that    $x$ belongs to $\Omega$, 
 and  the proof is complete.  
\end{proof}

 \subsection{The conservative case} 
 
 In the   conservative  
 case we can only prove Conjecture \ref{conj:topological}  in the case of a  Hölder conjugacy. 

\begin{thm}\label{thm:topological conservative}
Let $f_0$ and $f_1$ be two polynomial automorphisms of $\cd$ with non-trivial dynamics, and assume that $f_0$ is hyperbolic and 
$f_1$ is conservative.

 Suppose that there exists  respective  neighborhoods $N_0$ and $N_1$ 
 of $J_0 = \jstar_0$ and $\jstar_1$ and a Hölder continuous 
  homeomorphism $\phi: N_0\to N_1$ such that $\phi\circ f_0 = f_1\circ \phi$ where 
these  compositions makes sense. Then $f_1$ is hyperbolic. 
 \end{thm}
 
 The proof  is identical to that of Theorem \ref{thm:topological conservative}, the only difference is that in Lemma 
 \ref{lem:tube Pesin} we need a different argument to show that any ergodic invariant measure $\nu$ for $f_0$ admits a negative 
 Lyapunov exponent (this issue already appeared in the proof of Proposition \ref{prop:periodic}).  
 So Theorem \ref{thm:topological conservative} follows from:
 
 \begin{lem}
 Let $f_0$ and $f_1$ be as in Theorem \ref{thm:topological conservative}. 
 Then all  measures invariant under  $f_1$ are hyperbolic. 
 \end{lem}

\begin{proof}
This follows from standard Pesin-theoretic considerations. Let $\nu$ be an invariant measure for $f_1$. Without loss of generality we 
can assume that $\nu$ is ergodic so it admits two Lyapunov exponents $\chi_1\leq \chi_2$ with $\chi_1+\chi_2 = 0$. 
 Assume by way of  contradiction that $\chi_1 = \chi_2=0$.
   The Oseledets-Pesin reduction theorem (see \cite[Thm. S.2.10]{km}, note that it does not 
 require $\nu$ to be hyperbolic) asserts that for every $\e>0$ 
 there exists a measurable cocycle $C_\e$ 
 with values in $\mathrm{GL}_2(\cc)$ such that for $\nu$-a.e. $x$, the matrix 
 $A_\e (x):= C_\e(f_1(x))\inv \cdot  (Df_1)_x \cdot C_\e(x)$ satisfies $ e^{-\e}\leq \norm{A_\e(x)}\leq e^{\e}$  and
  $ e^{-\e}\leq \norm{(A_\e(x))\inv }\leq e^{\e}$. Then, the 
 Pesin  theorem on existence of regular neighborhoods  (see \cite[Thm. S.3.1]{km})
 implies  that there 
 is a measurable function $q$ such that for $\nu$-a.e. $x$, $f$ behaves likes $(Df_1)_x$ on $B(x, q(x))$ and furthermore 
 $e^{-\e}< q(f_1(x))/q(x) <e^{\e}$. More precisely there is a change of coordinates $\Psi_x$ defined on $B(x, q(x))$ such 
 that $ \Psi_{f_1(x)}\circ f \circ \Psi_x$  is $\e$ $C^1$-close to its differential at $x$, which equals   $A_\e (x)$. 
 
 Now by the Hölder conjugacy to $f_0$, for every $x$ there exists $y$ close to $x$ such that 
 $\dist(f_1^n(y), f_1^n(x))$ decreases like $e^{-\alpha n}$ for some $\alpha >0$. If we pick   $\e$ 
 small as  compared to   $\alpha$,     then for 
 generic $x$ we have that 
 $f_1^n(y) \in  B(f^n(x) , q(f^n(x))$ for every large $n$.  It follows that for large $k$
 $$\dist(f_1^{n+k} (y), f_1^{n+k}(x))\geq C e^{-2\e k}  \dist(f_1^n(y), f_1^n(x)),$$ which is contradictory, and the proof is complete.  
 \end{proof}

\end{document}